\documentclass[11pt]{article}
\usepackage{amsmath,amssymb,amsbsy,amsfonts,amsthm,latexsym,
               amsopn,amstext,amsxtra,euscript,amscd,color}
\usepackage[utf8]{inputenc}

\def \Q {{\mathbb Q}}
\def \Z {{\mathbb Z}}

\def \K {{\mathbb K}}

\def \T {{\rm Tr}}

\newtheorem{theorem}{Theorem}

\newtheorem{lemma}[theorem]{Lemma}

\newtheorem{corollary}[theorem]{Corollary}

\def\cO{{\mathcal O}}

\def\00{{\bf 0}}
\def\11{{\bf 1}}
\def\+{\oplus}

\def \Q {{\mathbb Q}}
\def \Z {{\mathbb Z}}

\def \T {{\rm Tr}}

\newcommand{\BBR}{\mathbb{R}}

\newcommand\norm{{\mathcal N}}

\newcommand\ph\varphi

\newcommand{\Gal}{{\mathrm{Gal}}}

\begin{document}

\title{\bf\huge On a divisibility relation for Lucas sequences}



\author{\Large Yuri F. Bilu$^1$\footnote{Partially supported by the Max-Planck Institut fur Mathematik}, \stepcounter{footnote}\stepcounter{footnote}
Takao Komatsu$^2$\footnote{Partially supported by the Hubei provincial Expert Program in China.},
Florian Luca$^3$\and \Large   Amalia Pizarro-Madariaga$^4$, Pantelimon~St\u anic\u a$^5$\footnote{Also Associated to the {\em Institute of Mathematics ``Simion Stoilow'' of the
 Romanian Academy}, Bucharest, Romania}\vspace{.5cm}\\
 $^1$ IMB, Universit\'e Bordeaux 1 \& CNRS, 351 cours de la Lib\'eration,\\
33405 Talence France;\\
Email: {\tt yuri@math.u-bordeaux.fr}\\
 $^2$ School of Mathematics and Statistics, Wuhan University,\\
Wuhan 430072 China;\\
Email: {\tt komatsu@whu.edu.cn}\\
 $^3$ School of Mathematics, University of the Witwatersrand, \\
 Private Bag X3, Wits 2050, South Africa; \\
 Email: {\tt florian.luca@wits.ac.za}\\
 $^4$ Instituto de Matem\'aticas,\\
Universidad de Valparaiso, Chile;\\
Email: {\tt amalia.pizarro@uv.cl}\\
$^5$ Naval Postgraduate School,  Applied Mathematics Department,\\
 Monterey, CA 93943--5216, USA;\\
 Email: {\tt pstanica@nps.edu}}

\maketitle

\begin{abstract}
In this note, we study the divisibility relation $U_m\mid U_{n+k}^s-U_n^s$, where ${\bf U}:=\{U_n\}_{n\ge 0}$ is the Lucas sequence of characteristic polynomial $x^2-ax\pm 1$ and $k,m,n,s$ are  positive integers.
\end{abstract}

\section{Introduction}
\label{Int}

Let ${\bf U}:={\bf U}(a,b)=\{U_n\}_{n\ge 0}$ be the Lucas sequence given by $U_0=0,~U_1=1$ and
\begin{equation}
\label{eq:Un}
U_{n+2}=aU_{n+1}+bU_n\quad {\text{\rm for all}}\quad n\ge 0,\quad {\text{\rm where}}\quad b\in \{\pm 1\}.
\end{equation}
 Its characteristic equation is $x^2-ax-b=0$ with roots
\begin{equation}
\label{eq:alphabeta}
(\alpha,\beta)=\left(\frac{a+{\sqrt{a^2+4b}}}{2}, \frac{a-{\sqrt{a^2+4b}}}{2}\right).
\end{equation}
When $a\ge 1$, we have that $\alpha>1>|\beta|$.
We assume that $\Delta=a^2+4b>0$ and that $\alpha/\beta$ is not a root of unity. This only excludes the pairs $(a,b)\in\{(0,\pm 1),~(\pm 1,-1),~(2,-1)\}$ from the subsequent considerations. Here, we look at the relation
\begin{equation}
\label{eq:Untos}
U_m\mid U_{n+k}^s-U_n^s,
\end{equation}
with positive integers $k,~m,~n,~s$. Note that when $(a,b)=(1,1)$, then $U_n=F_n$ is the  $n$th Fibonacci number. Taking $k=1$ and using the relations
\begin{eqnarray*}
F_{n+1}-F_n & = & F_{n-1},\\
F_{n+1}+F_n & = & F_{n+2},\\
F_{n+1}^2+F_n^2 & = & F_{2n+1},
\end{eqnarray*}
it follows that relation \eqref{eq:Untos} holds with $s=1,2,4$, and $m=n-1,~n+1,~2n+1$, respectively. Further, in \cite{KLT}, the authors assumed that $m$ and $n$ are coprime positive integers. In this case, $F_n$ and $F_m$ are coprime,
so the rational number $F_{n+1}/F_n$ is defined modulo $F_m$. Then it was shown in \cite{KLT} that if this last congruence class above has multiplicative order $s$ modulo $F_m$ and $s\not\in \{1,2,4\}$, then
\begin{equation}
\label{eq:m}
m<500 s^2.
\end{equation}
In this paper, we study the general divisibility relation \eqref{eq:Untos} and prove the following result.

\begin{theorem}
\label{thm:main}
Let~$a$ be a non-zero integer, ${b\in \{\pm1\}}$, and~$k$ a positive integer.
Assume that ${(a,b)\notin \{(\pm 1,-1),~(\pm 2,-1)\}}$. Given a positive integer~$m$, let~$s$ be the smallest positive integer such that divisibility \eqref{eq:Untos} holds. Then either ${s\in \{1,2,4\}}$, or
\begin{equation}
\label{eq:mainthm_new}
m<20000 (sk)^2.
\end{equation}
\end{theorem}

\section{Preliminary results}

We put ${\bf V}:={\bf V}(a,b)=\{V_n\}_{n\ge 0}$ for the Lucas companion of ${\bf U}$ which has initial values $V_0=2,~V_1=a$ and satisfies the same recurrence relation $V_{n+2}=a V_{n+1}+bV_n$ for all $n\ge 0$.
The Binet formulas for $U_n$ and $V_n$ are
\begin{equation}
\label{eq:Binet}
U_n=\frac{\alpha^n-\beta^n}{\alpha-\beta},\qquad  V_n=\alpha^n+\beta^n\qquad {\text{\rm for~all}}\qquad n\ge 0.
\end{equation}
The next result addresses the period of $\{U_n\}_{n\ge 0}$ modulo $U_m$, where $m\ge 1$ is fixed.

\begin{lemma}
\label{lem:period1}
The congruence
\begin{equation}
\label{eq:period}
U_{n+4m}\equiv U_n \pmod {U_m}
\end{equation}
holds for all $n\ge 0$, $m\geq 2$.
\end{lemma}

\begin{proof}
This follows because of the identity
$$
U_{n+4m}-U_n=U_mV_m V_{n+2m},
$$
 which can be easily checked using the Binet formulas \eqref{eq:Binet}.
\end{proof}

The following is Lemma 1 in \cite{KLT}. It has also appeared in other places.

\begin{lemma}
\label{lem:1}
Let $X\ge 3$ be a real number. Let $a$ and $b$ be positive integers with $\max\{a,b\}\le X$. Then there exist integers $u,v$ not both zero with $\max\{|u|,|v|\}\le
{\sqrt{X}}$ such that $|au+bv|\le 3{\sqrt{X}}$.
\end{lemma}

The following lemma is well-known, but we include the proof for the reader's convenience. In what follows, a unit means Dirichlet unit, that is an algebraic integer $\eta$ such that $\eta^{-1}$ is also an algebraic integer.

\begin{lemma}
\label{lem:v_new}
Let ${v>1}$ be an integer and~$\zeta$ be a primitive $v$th root of unity. Then
\begin{equation}
\label{eq:v_new}
\prod_{\gcd(k,v)=1}(1-\zeta^k)=
\begin{cases}
p, & \text{if ${v=p^\ell}$ is a prime power},\\
1, &\text{if $v$ has at least two distinct prime divisors},
\end{cases}
\end{equation}
the product being over the residue classes $\bmod~v$ coprime with~$v$. In particular, in the second case, ${1-\zeta}$ is a unit.
\end{lemma}

\begin{proof}
The product on the left of~\eqref{eq:v_new} is $\Phi_v(1)$, where $\Phi_v(X)$ denotes the $v$th cyclotomic polynomial. For ${v=p^\ell}$ we have
$$
\Phi_{p^\ell}(X)= \frac{X^{p^\ell}-1}{X^{p^{\ell-1}}-1}= X^{p^{\ell-1}(p-2)}+X^{p^{\ell-1}(p-1)}+\cdots+X^{p^{\ell-1}}+1,
$$
and ${\Phi_{p^\ell}(1)=p}$, proving the prime power case. In particular, ${(1-\zeta)\mid p}$ in this case.

Now assume that~$v$ is divisible by two distinct primes~$p$ and~$p'$. Then $\zeta^{v/p}$ is a primitive root of unity of order~$p$, which implies that in the ring $\Z[\zeta]$ we have ${(1-\zeta)\mid(1-\zeta^{v/p})\mid p}$. Similarly, ${(1-\zeta)\mid p'}$. The divisibility relations ${(1-\zeta)\mid p}$ and ${(1-\zeta)\mid p'}$ imply that ${(1-\zeta)\mid 1}$, that is, ${1-\zeta}$ is a unit. Hence its ${\Q(\zeta)/\Q}$-norm is $\pm1$. Since it is obviously positive, it must be~$1$.  But this norm is exactly the left-hand side of~\eqref{eq:v_new}.
\end{proof}


This lemma has the following consequence, which is again   well-known, but we did not find a suitable reference.

\begin{corollary}
\label{cor:v_new}
\begin{enumerate}
\item
\label{item:zetaxi}
Assume that~$\zeta$ and~$\xi$ are roots of unity of coprime orders, and both distinct from~$1$. The ${\zeta-\xi}$ is a unit.

\item[] From now on~$m$ and~$n$ are positive integers and ${d=\gcd(m,n)}$.

\item
\label{item:mn}
In $\Z[x]$ we have the equality of ideals ${(x^m-1,x^n-1)=(x^d-1)}$.

\item
\label{item:gamma}
Let~$\gamma$ be an algebraic integer in some number field~$K$. Then we have the equality of  $\cO_K$-ideals
${(\gamma^m-1,\gamma^n-1)=(\gamma^d-1)}$.
\end{enumerate}
\end{corollary}

\begin{proof}
Item~\ref{item:zetaxi} follows from the second assertion of Lemma~\ref{lem:v_new}.

In item~\ref{item:mn} it suffices to show that ${x^d-1\in (x^m-1,x^n-1)}$. In the case ${d=1}$ this reduces to showing that
\begin{equation}
\label{eq:res}
1\in \left(\frac{x^m-1}{x-1},\frac{x^n-1}{x-1}\right).
\end{equation}
The resultant of the polynomials $\frac{x^m-1}{x-1}$ and $\frac{x^n-1}{x-1}$ is the product of the factors of the form ${\zeta-\xi}$, where~$\zeta$ and~$\xi$ are roots of unity of orders  dividing~$m$ and~$n$, respectively, and none of~$\zeta,\xi$ is~$1$. If ${d=\gcd(m,n)=1}$, then each factor is a unit by item~\ref{item:zetaxi}. Hence, the resultant is a unit of~$\Z$, that is, $\pm1$, proving~\eqref{eq:res} in the case ${d=1}$.

The case of arbitrary~$d$ reduces to the case ${d=1}$. Indeed, by the latter, ${x^d-1}$ belongs to the ideal ${(x^m-1,x^n-1)}$ in the ring $\Z[x^d]$. Hence, the same is true in the ring $\Z[x]$.

Finally, item~\ref{item:gamma} is an immediate consequence of the previous item.
\end{proof}

We will use one simple property of cyclotomic polynomials. Recall that for a positive integer~$v$ we denote by $\Phi_v(X)$ the $v$th cyclotomic polynomial. Then for ${\alpha>1}$ we have the trivial estimate ${\Phi_v(\alpha)>(\alpha-1)^{\ph(v)}}$ (where $\ph(v)$ is, of course, the Euler totient). We will need a slightly sharper estimate.

\begin{lemma}
\label{lem:cyclo}
Let~$v$ be a positive integer and  ${\alpha>1}$ a real number. Then for ${v>1}$   we have
\begin{equation}
\label{eq:cyclo}
\Phi_v(\alpha)> \bigl(\alpha(\alpha-1)\bigr)^{\ph(v)/2}.
\end{equation}
\end{lemma}

\begin{proof}
We use the identity
$$
\Phi_v(X)=\prod_{d\mid v}(X^d-1)^{\mu(v/d)},
$$
where $\mu(\cdot)$ is the M\"obius function. We have clearly
\begin{equation}
\label{eq:mu}
(\alpha^d-1)^{\mu(v/d)} \ge
\begin{cases}
\alpha^{d\mu(v/d)} , & \mu(v/d)=-1, \\
\alpha^{d\mu(v/d)}\frac{\alpha-1}\alpha , & \mu(v/d)=1.
\end{cases}
\end{equation}
Moreover:
\begin{itemize}
\item
denoting by $\tau^\ast(v)$  the number of square-free divisors of~$v$, we have, for ${v>1}$, exactly $\tau^\ast(v)/2$ divisors with ${\mu(v/d)=1}$ and exactly $\tau^\ast(v)/2$ divisors with ${\mu(v/d)=-1}$;
\item
inequality~\eqref{eq:mu} is strict for all   ${d\mid v}$ satisfying ${\mu(v/d)\ne 0}$, with at most one exception.
\end{itemize}
Hence, multiplying~\eqref{eq:mu} for all  $d\mid v$ with ${\mu(v/d)\ne 0}$, and using the identity ${\ph(v)=\sum_{d\mid v}d\mu(v/d)}$,  we obtain, for ${v>1}$, the lower estimate
\begin{equation}
\label{eq:taustar}
\Phi_v(\alpha)> \alpha^{\ph(v)}\left(\frac{\alpha-1}\alpha\right)^{\tau^\ast(v)/2}.
\end{equation}
For ${v\notin \{1,2,6\}}$,  we have ${\tau^\ast(v)\le \ph(v)}$, which  implies
$$
|\Phi_v(\alpha)|> \alpha^{\ph(v)}\left(\frac{\alpha-1}\alpha\right)^{\ph(v)/2}=\bigl(\alpha(\alpha-1)\bigr)^{\ph(v)/2},
$$
proving~\eqref{eq:cyclo} for ${v\notin \{1,2,6\}}$. And for ${v\in\{2,6\}}$, this is obviously true.
\end{proof}

The following lemma is the workhorse of our argument.

\begin{lemma}
\label{lem:independence}
Let~$a$,~$b$ and~$k$ be as in the statement of Theorem~\textup{\ref{thm:main}}, and assume in addition that $ {a\ge 1}$. Let ${v\ge 1}$ be an integer and~$\zeta$ a primitive $v$th root of unity. Define~$\alpha$ as in~\eqref{eq:alphabeta} and assume  that the numbers
\begin{equation}
\label{eq:2numbers}
\alpha\qquad {\text{\it and}}\qquad \frac{\alpha^k-(-b)^k \bar\zeta}{\alpha^k-\zeta}
\end{equation}
are multiplicatively dependent. Then we have one of the following options:
\begin{itemize}
\item[$(i)$] $(-b)^k=-1$, $v=4$;
\item[$(ii)$] $(a,b,k)\in\{(1,1,1),~(2,1,1)\}$ and $v\in \{1,2\}$;
\item[$(iii)$] $(-b)^k=1$, $v\in \{1,2\}$;
\item[$(iv)$] $(a,b,k)=(4,-1,1)$ and $v\in \{4,6\}$.
\end{itemize}
\end{lemma}

\begin{proof}
We use the notation
$$
K=\Q(\alpha), \quad L=\Q(\zeta), \quad M=\Q(\alpha,\zeta), \quad \alpha_1=\alpha^k, \quad \delta=(-b)^k.
$$
Note that $\delta\alpha_1^{-1}=\beta^k$ is the Galois conjugate of $\alpha_1$.

Recall that we disregard the cases ${(a,b)\in \{(1,-1),(2,-1)\}}$. In addition to this, we will disregard the case ${(a,b,k)=(1,1,1)}$, because it is settled in Lemma~2 of~\cite{KLT}.  This implies that
\begin{equation}
\label{eq:lower}
\alpha_1\ge 1+\sqrt2.
\end{equation}
When ${\delta= 1}$ we can say more:
\begin{equation}
\label{eq:lowersharp}
\alpha_1\in \left\{ \frac{3+\sqrt5}2, 2+\sqrt3\right\} \quad \text{or}\quad \alpha_1\ge\frac{5+\sqrt{21}}2 .
\end{equation}
We will also assume that we are not in one of the instances~$(i)$,~$(iii)$ above; this is equivalent to saying that
\begin{equation}
\label{eq:nzero}
\zeta^2\ne \delta.
\end{equation}

Since the numbers~\eqref{eq:2numbers} are multiplicatively dependent, then the second of these numbers must be a unit (because the first is). In particular,
$$
{(\alpha_1-\zeta)\mid (\alpha_1-\delta \bar\zeta)}
$$
in the ring $\cO_M$, which implies that
\begin{equation}
\label{eq:divides}
(\alpha_1-\zeta)\mid(\zeta-\delta\bar\zeta).
\end{equation}
This divisibility relation is very restrictive: we will see that is satisfied in very few cases, which can be verified by inspection.

We first show the following   identity for the norm of $\alpha_1-\zeta$:
\begin{equation}
\label{eq:leftnorm}
|\norm_{M/\Q}(\alpha_1-\zeta)|=\bigl(\alpha_1^{-\ph(v)}\Phi_v(\alpha_1)\Phi_{v^\ast}(\alpha_1)\bigr)^{[M:L]/2},
\end{equation}
where $\Phi_v(X)$ is the $v$th cyclotomic polynomial and
\begin{equation}
\label{eq:*}
v^\ast=
\begin{cases}
v & \text{if $4\mid v$ or $\delta=1$},\\
v/2& \text{if $2\,\|\,v$ and $\delta=-1$},\\
2v& \text{if $2\nmid v$ and $\delta=-1$}.
\end{cases}
\end{equation}
Note that
$$
\ph(v^\ast)=\ph(v), \qquad \Phi_{v^\ast}(X)=\pm \Phi_v(\delta X), \qquad \Phi_v(X^{-1})=\pm X^{-\ph(v)}\Phi_v(X),
$$
the sign in last identity being ``$+$'' for ${v>1}$ and the sign in the middle identity being ``$+$" if $\delta=1$ or  $\min\{v,v^*\}>1$.

Let us prove~\eqref{eq:leftnorm}.  When ${\alpha \notin L}$, the conjugates of ${\alpha_1-\zeta}$ are the $2\ph(v)$ numbers ${\alpha_1-\zeta'}$ and ${\delta\alpha_1^{-1}-\zeta''}$, where both~$\zeta'$ and~$\zeta''$ run through the set of primitive $v$th roots of unity. Hence, in this case
$$
|\norm_{M/\Q}(\alpha_1-\zeta)|=|\Phi_v(\alpha_1)\Phi_v(\delta\alpha_1^{-1})|= \alpha_1^{-\ph(v)}\Phi_v(\alpha_1)\Phi_{v^\ast}(\alpha_1),
$$
which is~\eqref{eq:leftnorm} in the case ${\alpha \notin L}$.

Now assume that ${\alpha \in L}$, and set
$$
G=\Gal(L/\Q), \qquad H=\Gal(L/K),
$$
for the Galois groups of the various extensions.
The group~$H$ is a subgroup of~$G$ of index~$2$,
and we have
$$
\alpha_1^\sigma=
\begin{cases}
\alpha_1, & \sigma \in H,\\
\delta\alpha_1^{-1}, &\sigma \in G\smallsetminus H.
\end{cases}
$$
Hence,
\begin{align*}
|\norm_{M/\Q}(\alpha_1-\zeta)|& =|\norm_{L/\Q}(\alpha_1-\zeta)|\\
&=\prod_{\sigma\in H}|\alpha_1-\zeta^\sigma|\prod_{\sigma\in G\smallsetminus H}|\delta\alpha_1^{-1}-\zeta^{\sigma}|\\
&=\alpha_1^{-\ph(v)/2}\prod_{\sigma\in H}|\alpha_1-\zeta^\sigma|\prod_{\sigma\in G\smallsetminus H}|\delta\alpha_1-\zeta^{\sigma}|,
\end{align*}
where in the second equality we used ${\alpha_1\in \BBR}$.
In a similar fashion,
\begin{align*}
|\norm_{M/\Q}(\alpha_1-\delta\bar\zeta)| &=\prod_{\sigma\in H}|\alpha_1-\delta\bar\zeta^\sigma|\prod_{\sigma'\in G\smallsetminus H}|\delta\alpha_1^{-1}-\delta\bar\zeta^{\sigma}|\\
&=\alpha_1^{-\ph(v)/2}\prod_{\sigma\in H}|\delta\alpha_1-\zeta^\sigma|\prod_{\sigma\in G\smallsetminus H}|\alpha_1-\zeta^{\sigma}|.
\end{align*}
Since ${\frac{\alpha_1-\delta\bar\zeta}{\alpha_1-\zeta}}$ is a unit, the two norms computed above are equal. Hence,
\begin{align*}
|\norm_{M/\Q}(\alpha_1-\zeta)|^2& =|\norm_{M/\Q}(\alpha_1-\zeta)\norm_{M/\Q}(\alpha_1-\delta\bar\zeta)|\\
&=\alpha_1^{-\ph(v)}\prod_{\sigma\in G}|\alpha_1-\zeta^\sigma|\prod_{\sigma\in G}|\delta\alpha_1-\zeta^{\sigma}|\\
&=\alpha_1^{-\ph(v)}\Phi_v(\alpha_1)\Phi_{v^\ast}(\alpha_1),
\end{align*}
which proves~\eqref{eq:leftnorm} in the case ${\alpha \in L}$ as well.

Combining~\eqref{eq:divides} and~\eqref{eq:leftnorm} and recalling~\eqref{eq:nzero}, we obtain  the inequality
\begin{equation}
\label{eq:inequality}
0<\alpha_1^{-\ph(v)/2}|\Phi_v(\alpha_1) \Phi_{v^\ast}(\alpha_1)|^{1/2} \le |\norm_{L/\Q}(1-\delta\zeta^2)|.
\end{equation}
This will be our basic tool.

Our next observation is that ${1-\delta\zeta^2}$ cannot be a unit. Indeed, if it is a unit, then the right-hand side of~\eqref{eq:inequality} is~$1$ and  $\min\{v,v^\ast\}>1$. Hence, applying Lemma~\ref{lem:cyclo}, we obtain
$$
\alpha_1^{-\ph(v)/2}\bigl(\alpha_1(\alpha_1-1)\bigr)^{\ph(v)/2}<1,
$$
which implies ${\alpha_1<2}$, contradicting~\eqref{eq:lower}.

Thus, ${1-\delta\zeta^2}$ is non-zero, but not a unit. Applying Lemma~\ref{lem:v_new}, we find that this is possible only in the following cases:
\begin{align}
\label{eq:pell}
v&=p^\ell, &\delta&=1,\\
\label{eq:twopell}
v&=2p^\ell, &\delta&=1,\\
\label{eq:twoell}
v&=2^\ell, &\ell&\ge3,\\
\label{eq:onetwofour}
v&\in\{1,2,4\}, &\delta&\ne \zeta^2,
\end{align}
where (here and below)~$\ell$ is a positive integer and~$p$ is an odd prime number. We study these cases separately.

In the case~\eqref{eq:pell}, we have
$$
{\Phi_v(X)=\Phi_{v^\ast}(X)=\frac{X^{p^\ell}-1}{X^{p^{\ell-1}}-1}}\qquad {\text{\rm and}}\qquad {\norm_{L/\Q}(1-\zeta^2)=p}
$$
by Lemma~\ref{lem:v_new}. We obtain
$$
\frac1{\alpha_1^{p^{\ell-1}(p-1)/2}}\frac{\alpha_1^{p^\ell}-1}{\alpha_1^{p^{\ell-1}}-1} \le p.
$$
The left-hand side is strictly bounded from below
by ${\alpha^{p^{\ell-1}(p-1)/2}}$, which gives
${\alpha_1^{p^{\ell-1}}< p^{\frac2{p-1}}}$.
Checking with~\eqref{eq:lowersharp} leaves the only option
$$
\alpha_1=\frac{3+\sqrt5}2, \qquad p^\ell=3,
$$
which is eliminated by direct verification.

In the case~\eqref{eq:twopell}, we have
$$
{\Phi_v(X)=\Phi_{v^\ast}(X)=\frac{X^{p^\ell}+1}{X^{p^{\ell-1}}+1}}\qquad {\text{\rm and}}\qquad {\norm_{L/\Q}(1-\zeta^2)=p}.
$$
We obtain
$$
\frac1{\alpha_1^{p^{\ell-1}(p-1)/2}}\frac{\alpha_1^{p^\ell}+1}{\alpha_1^{p^{\ell-1}}+1} \le p.
$$
From~\eqref{eq:lowersharp}, we deduce    ${\alpha_1^{p^{\ell-1}}+1\le 1.4\,\alpha_1^{p^{\ell-1}}}$, which  implies the inequality
${\alpha_1^{p^{\ell-1}}< (1.4\,p)^{\frac2{p-1}}}$.
Invoking again~\eqref{eq:lowersharp}, we are left with the three options
\begin{align}
\label{eq:bads}
\alpha_1&=\frac{3+\sqrt5}2, &p^\ell&\in \{3,5\}, \\
\label{eq:good}
\alpha_1&=2+\sqrt3,&p^\ell&=3.
\end{align}
The two cases in~\eqref{eq:bads} are eliminated by verification, while~\eqref{eq:good} leads to ${(a,b,k,v)=(4,-1,1,6)}$,  one of the two  instances in~($iv$).


In the case~\eqref{eq:twoell}, we have
$$
{\Phi_v(X)=\Phi_{v^\ast}(X)=\frac{X^{2^\ell}-1}{X^{2^{\ell-1}}-1}}\qquad {\text{\rm and}}\qquad {\norm_{L/\Q}(1-\delta\zeta^2)=4}.
$$
We obtain
$$
\frac1{\alpha_1^{2^{\ell-2}}}\frac{\alpha_1^{2^\ell}-1}{\alpha_1^{2^{\ell-1}}+1} \le 4,
$$
which implies ${\alpha_1^{2^{\ell-2}}\le 4}$. Since ${\ell\ge 3}$, this contradicts~\eqref{eq:lower}.

In the final case~\eqref{eq:onetwofour}, it more convenient to use the divisibility relation~\eqref{eq:divides} directly. If  ${v\in\{1,2\}}$, then ${\zeta^2=1}$ and ${\delta=-1}$.
Taking the norm in both sides of~\eqref{eq:divides}, we obtain
$$
\alpha_1-\alpha_1^{-1}=\T_{K/\Q}(\alpha_1)\mid 4.
$$
Together with ${\norm_{K/\Q}(\alpha_1)=\delta=-1}$ and inequality~\eqref{eq:lower}, this implies two possibilities:
\begin{equation}
\label{eq:options}
\alpha_1=1+\sqrt2, \quad \alpha_1=2+\sqrt3.
\end{equation}
The latter is disqualified by inspection. The former yields ${(a,b,k)=(2,1,1)}$, which is~$(ii)$.

In a similar fashion one treats  ${v=4}$. In this case ${\zeta^2=-1}$ and ${\delta=1}$, and, taking the norm in~\eqref{eq:divides}, we obtain
$$
(\alpha_1+\alpha_1^{-1})^2=(\T_{K/\Q}(\alpha_1))^2\mid 16.
$$
We again have one of the options~\eqref{eq:options}, but this time the former is eliminated by inspection, and  the latter leads to ${(a,b,k)=(4,-1,1)}$, the missing instance in~$(iv)$.
This completes the proof of the lemma.
\end{proof}

The following is a generalization of Lemma 4 from \cite{KLT}.

For a prime number $p$ and a nonzero integer $m$, we put $\nu_p(m)$ for the exponent of the prime $p$ in the factorization of $m$. For a  finite set of primes
${\mathcal S}$ and a positive integer $m$, we put
$$
m_{\mathcal S}=\prod_{p\in {\mathcal S}} p^{\nu_p(m)}
$$
for the largest divisor of $m$ whose prime factors are in ${\mathcal S}$. For any prime number $p$ we put $f_p$ for the index of appearance in the Lucas sequence $\{U_n\}_{n\ge 0}$, which is the minimal positive integer $k$ such that $p\mid U_{k}$.

\begin{lemma}
\label{lem:4}
Let $a\ge 1$. If ${\mathcal S}$ is any finite set of primes and $m$ is a positive integer, then
$$
(U_m)_{\mathcal S}\le \alpha^2 m\, {\text{\rm lcm}}[ U_{f_p}: p\in {\mathcal S}].
$$
\end{lemma}

\begin{proof}
It is known that
$$
\nu_p(U_m)=\left\{
\begin{matrix}
0 & {\text{\rm if}} & m\not\equiv 0\pmod {f_p}; & \\
\nu_p(U_{f_p})+\nu_p(m/f_p) & {\text{\rm if}} & m\equiv 0\pmod {f_p}, & p~{\text{\rm odd}};\\
\nu_2(U_2)+\nu_2(m/2) & {\text{\rm if}} & m\equiv 0\pmod {2}, & p=2,~a~{\text{\rm even}};\\
\nu_2(U_3) & {\text{\rm if}} & m\equiv 3\pmod 6, & p=2,~a~{\text{\rm odd}};\\
\nu_2(U_6)+\nu_2(m/2) & {\text{\rm if}} & m\equiv 0\pmod 6, & p=2,~a~{\text{\rm odd}}.
\end{matrix}
\right.
$$
The above relations follow easily from Proposition 2.1 in \cite{BHV}. In particular, the inequality
$$
\nu_p(U_m)\le \nu_p(U_{f_p})+\nu_p(m)+\delta_{p,2}
$$
always holds with $\delta_{p,2}$ being $0$ if $p$ is odd or $p=2$ and $a$ is even and $\nu_2((a^2+3b)/2)$ if $p=2$ and $a$ is odd. We get that
\begin{eqnarray*}
(U_m)_{\mathcal S} & \le &  \left(\prod_{p\in {\mathcal S}} p^{\nu_p(U_{f_p})}\right)\left(\prod_{\substack{p\mid m\\ p>2}} p^{\nu_p(m)}\right) 2^{\nu_2(m)+\nu_2((a^2+3b)/2}
\\
& < & \alpha^2 m\, {\text{\rm lcm}} [U_{f_p}: p\in {\mathcal S}],
\end{eqnarray*}
which is what we wanted to prove. For the last inequality above, we used the fact that $2^{\nu_2((a^2+3b)/2)}\le (a^2+3b)/2=(\alpha^2-\alpha\beta+\beta^2)/2< \alpha^2$.
\end{proof}

\section{Proof of Theorem \ref{thm:main}}

We assume that $m\ge 10000k$. Since $U_n$ is periodic modulo $U_m$ with period $4m$ (Lemma \ref{lem:period1}), we may assume that $n\le 4m$.
We split $U_m$ into various factors, as follows. Write
$$
U_{n+k}^s-U_n^s=\prod_{d\mid s} \Phi_d(U_{n+k},U_n), 
$$
where $\Phi_d(X,Y)$ is the homogenization of the cyclotomic polynomial $\Phi_d(X)$. We put $s_1:={\text{\rm lcm}}[2,s]$, ${\mathcal S}:=\{p: p\mid 6s\}$ and
\begin{eqnarray*}
D & := & (U_m)_{\mathcal S};\\
A & : = & \gcd(U_m/D, \prod_{d\le 6, ~d\ne 5} \Phi_d(U_{n+k},U_n);\\
E & := & \gcd(U_m/D, \prod_{\substack{d\mid s_1\\ d=5~{\text{\rm or}}~d>6}} \Phi_d(U_{n+k},U_n).
\end{eqnarray*}
Clearly,
$$
U_m\mid ADE.
$$
Before bounding $A,~D,~E$, let us  comment on the sign of $a$. If $a<0$, then we change the pair $(a,b)$ to $(-a,b)$. This has as effect
replacing $(\alpha,\beta)$ by $(-\alpha,-\beta)$ and so $U_n(\alpha,\beta)=(-1)^{n-1} U_n(\alpha,\beta)$ for all $n\ge 0$.  In particular, $U_m$ remains the same or changes sign.
Further, if $k$ is even then
$$
\Phi_d(U_{n+k}(-\alpha,-\beta),U_n(-\alpha,-\beta))=\pm \Phi_d(U_{n+k}(\alpha,\beta), U_n(\alpha,\beta)),
$$
while if $k$ is odd, then
\begin{eqnarray*}
\Phi_d(U_{n+k}(-\alpha,-\beta),U_n(-\alpha,-\beta)) & = & \pm \Phi_d(U_{n+k}(\alpha,\beta),-U_n(\alpha,\beta))\\
& = & \pm \Phi_{d^\ast}(U_{n+k}(\alpha,\beta),U_n(\alpha,\beta)),
\end{eqnarray*}
where the $d^\ast$ has been defined at \eqref{eq:*}. Note that the sets $\{d\le 6, ~d\ne 5\}$ and $\{d\mid s_1,~d=5~{\text{\rm or}}~d>6\}$ are closed under the operation $d\mapsto d^\ast$.
Hence, $D,~A,~E$ do not change if we replace $a$ by $-a$, so we assume that $a>0$. By the Binet formula \eqref{eq:Binet}, we get easily that the inequality
\begin{equation}
\label{eq:lowup}
\alpha^{n-2}\le U_n\le \alpha^{n}\quad {\text{\rm is~valid~for~all}}\quad n\ge 1.
\end{equation}
We are now ready to bound $A,~D,~E$.

The easiest to bound is $D$.  Namely, by Lemma \ref{lem:4} and the fact that  $f_p\le p+1$ for all $p\mid 6s$, we get
\begin{equation}
\label{eq:boundforD}
D\le \alpha^2 m  \prod_{p\mid 6s} U_{p+1}<m \alpha^{2+\sum_{p\mid 6s} (p+1)}<\alpha^{6s+3+\log m/\log \alpha},
\end{equation}
where we used the fact that $\sum_{p\mid t} (p+1)\le t+1$, which is easily proved by induction on the number of distinct prime factors of $t$.

\medskip

We next bound $E$.

\medskip

Note that
\begin{equation}
\label{eq:E}
E\mid \prod_{\substack{\zeta: \zeta^{s_1}=1\\ \zeta\not\in \{\pm 1,\pm i,\pm \omega,\pm \omega^2\}}} (U_{n+k}-\zeta U_n),
\end{equation}
where $\omega:=e^{2\pi i/3}$ is a primitive root of unity of order $3$.

\medskip

Let $K=\Q(e^{2\pi i/s_1}, \alpha)$, which is a number field of degree $d\le 2\phi(s_1)=2\phi(s)$.
Assume that there are $\ell$ roots of unity $\zeta$ participating in the product appearing in the right--hand side of \eqref{eq:E} and label them $\zeta_1,\ldots,\zeta_{\ell}$.
Write
\begin{equation}
\label{eq:rel}
{\mathcal E}_i=\gcd(E,U_{n+k}-\zeta_i U_n)\quad {\text{for~all}}\quad i=1,\ldots,\ell,
\end{equation}
where ${\mathcal E}_i$ are ideals in ${\mathcal O}_{K}$. Then relations \eqref{eq:E} and \eqref{eq:rel} tell us that
\begin{equation}
\label{eq:prod}
E{\mathcal O}_{K}\mid \prod_{i=1}^{\ell} {\mathcal E}_i.
\end{equation}
Our next goal is to bound the norm $|\norm_{K/\Q}({\mathcal E}_i)|$ of ${\mathcal E}_i$ for $i=1,\ldots,\ell.$ First of all, $U_m\in {\mathcal E_i}$. Thus, with formula \eqref{eq:Binet}
and the fact that $\beta=(-b)\alpha^{-1}$, we get
$$
\alpha^m\equiv (-b)^m \alpha^{-m}\pmod {{\mathcal E}_i}.
$$
Multiplying the above congruence by $\alpha^m$, we get
\begin{equation}
\label{eq:T1}
\alpha^{2m}\equiv (-b)^m \pmod {{\mathcal E}_i}.
\end{equation}
We next use formulae \eqref{eq:Binet} and \eqref{eq:rel} to deduce that
$$
(\alpha^{n+k}-(-b)^{n+k} \alpha^{-n-k})-\zeta(\alpha^n-(-b)^n \alpha^{-n})\equiv 0\pmod {{\mathcal E}_i},\quad (\zeta:=\zeta_i).
$$
Multiplying both sides above by $\alpha^n$, we get
\begin{equation}
\label{eq:rel2}
\alpha^{2n}(\alpha^k-\zeta)-(-b)^{n+k}(\alpha^{-k}-(-b)^k\zeta)\equiv 0\pmod {{\mathcal E}_i}.
\end{equation}
Let us show that $\alpha^k-\zeta$ and ${\mathcal E}_i$ are coprime. Assume this is not so and let $\pi$ be some prime ideal of ${\mathcal O}_{\K}$ dividing both $\alpha^k-\zeta$ and ${\mathcal E}_i$.
Then we get $\alpha^k\equiv \zeta\pmod \pi$ and so $\alpha^{-k}\equiv (-b)^k \zeta\pmod \pi$ by \eqref{eq:rel2}.
Multiplying these two congruences we get $1\equiv (-b)^k \zeta^2\pmod \pi$. Hence, $\pi\mid 1-(-b)^k \zeta^2$. If this number is not zero, then, $(-b)^k \zeta^2$ is a root of unity whose order divides $6s$, so, by Lemma \ref{lem:cyclo}, we get that $\pi\mid 6s$, which is
impossible because $\pi\mid {\mathcal E}_i\mid E$, and $E$ is an integer coprime to $6s$. If the above number is zero, we get that $\zeta^2=\pm 1$, so $\zeta\in \{\pm 1,\pm i\}$, but these values are excluded at this step. Thus, indeed $\alpha^k-\zeta$ and ${\mathcal E}_i$ are coprime, so $\alpha^k-\zeta$ is invertible modulo ${\mathcal E}_i$.
Now congruence \eqref{eq:rel2} shows that
\begin{equation}
\label{eq:12}
\alpha^{2n+k}\equiv (-b)^{n} \zeta \left(\frac{\alpha^{k}-(-b)^k {\overline{\zeta}}}{\alpha^k-\zeta}\right)\pmod {{\mathcal E}_i}.
\end{equation}
We now apply Lemma \ref{lem:1} to $a=2m$ and $b=2n+k\le 8m +k<9m$ with the choice $X=9m$ to deduce that there exist integers $u$, $v$ not both zero with
$\max\{|u|,|v|\}\le {\sqrt{X}}$ such that $|2mu+(2n+k)v|\le 3{\sqrt{X}}.$ We raise congruence \eqref{eq:T1} to $u$ and congruence \eqref{eq:12} to
$v$ and multiply the resulting congruences getting
$$
\alpha^{2mu+(2n+k)v}=(-b)^{mu+nv} \zeta^{v} \left(\frac{\alpha^k-(-b)^k{\overline{\zeta}}}{\alpha^k-\zeta}\right)^v \pmod {{\mathcal E}_i}.
$$
We record this as
\begin{equation}
\label{eq:T3}
\alpha^{R}\equiv \eta \left(\frac{\alpha^k-(-b)^k{\overline{\zeta}}}{\alpha^k-\zeta}\right)^S\pmod {{\mathcal E}_i}
\end{equation}
for suitable roots of unity $\eta$ and $\zeta$ of order dividing $s_1$ with $\zeta$ not of order $1,2,3,4$, or $6$,
where $R:=2mu+(2n+k)v$ and $S:=v$. We may assume that $R\ge 0$, for if not, we replace the pair $(u,v)$ by the pair $(-u,-v)$, thus replacing $(R,S)$ by $(-R,-S)$ and $\eta$ by $\eta^{-1}$ and leaving $\zeta$ unaffected. We may additionally assume that
$S\ge 0$, for if not, we replace $S$ by $-S$ and $\zeta$ by $(-b)^k {\overline{\zeta}}$, again a root of unity of order dividing $s_1$ but not of order $1,2,3,4,$ or $6$ and leave $R$ and $\eta$ unaffected. Thus, ${\mathcal E}_i$ divides the algebraic integer
\begin{equation}
\label{eq:alginteger}
E_i=\alpha^R (\alpha^k-\zeta_i)^{S}-\eta_i(\alpha^k-(-b)^k{\overline{\zeta_i}})^S.
\end{equation}
Let us show that $E_i\ne 0$. If $E_i=0$, we then get
$$
\alpha^R=\eta_i \left(\frac{\alpha-(-b)^k{\overline{\zeta_i}}}{\alpha-\zeta_i}\right)^S,
$$
and after raising both sides of the above equality to the power $s_1$, we get, since $\eta_i^{s_1}=1$, that
$$
\alpha^{s_1R}=\left(\frac{\alpha^k-(-b)^k{\overline{\zeta_i}}}{\alpha-\zeta_i}\right)^{Ss_1}.
$$
Lemma \ref{lem:independence} gives us a certain number of conditions  all of which have $\zeta_i$ or a root of unity of order $1,2,4,$ or $6$, which is not our case.
Thus, $E_i$ is not equal to zero.
We now bound the absolute values of the conjugates of $E_i$. We find it more convenient to work with the associate of $E_i$ given by
$$
G_i=\alpha^{-\lfloor R/2\rfloor} E_i= \alpha^{R-\lfloor R/2\rfloor}  (\alpha^k-\zeta_i)^{S}-\alpha^{-\lfloor R/2\rfloor}\eta_i(\alpha^k-(-b)^k{\overline{\zeta_i}})^S.
$$
Note that
$$
R\le |2m+(2n+k)v|\le 3{\sqrt{X}}=9{\sqrt{m}},\quad {\text{\rm and}}\quad S=|v|\le {\sqrt{X}}=3{\sqrt{m}}.
$$
Let $\sigma$ be an arbitrary element of $G={\text{\rm Gal}}(K/\Q)$. We then have that $\eta_i^{\sigma}=\eta_i'$, $\zeta_i^{\sigma}=\zeta_i'$, where $\eta_i'$ and $\zeta_i'$
are roots of unity of order dividing $s_1$. Furthermore, $\alpha^{\sigma}\in \{\alpha,\beta\}$.
If $\alpha^{\sigma}=\alpha$, we then get
\begin{eqnarray}
\label{eq:conj}
|G_i^{\sigma}| & = & |\alpha^{R-\lfloor R/2\rfloor}(\alpha^k-\zeta_i')^{S}-\eta_i'\alpha^{-\lfloor R/2\rfloor}(\alpha-(-b)^k{\overline{\zeta_i'}})^S|\nonumber\\
& \le & \alpha^{(R+1)/2} (\alpha^k+1)^S+(\alpha^k+1)^S\nonumber\\
& \le &  2\alpha^{(R+1)/2} (\alpha+1)^{Sk}\le \alpha^{2+(9{\sqrt{m}}+1)/2+6{\sqrt{m}} k}\nonumber\\
& \le & \alpha^{11 {\sqrt{m}} k},
\end{eqnarray}
while if $\alpha^{\sigma}=\beta$, we also get
\begin{eqnarray*}
|G_i^{\sigma}| & = & |\beta^{R-\lfloor R/2\rfloor}(\beta^k-\zeta_i')^b-\beta^{-\lfloor R/2\rfloor }\eta_i'(\beta^k-(-b)^k{\overline{\zeta_i'}})^S|\nonumber\\
&\le & (\alpha^{-k}+1)^S+\alpha^{R/2}(\alpha^{-k}+1)^S\nonumber\\
& = &
\alpha^S+\alpha^{R/2+S}\le 2 \alpha^{R/2+S}\le \alpha^{2+4.5 {\sqrt{m}}+6{\sqrt{m}}} \\
& = & \alpha^{11 {\sqrt{m}} k}.
\end{eqnarray*}
In the above, we used the fact that $\alpha^{-k}+1\le \alpha^{-1}+1\le \alpha$.
In conclusion, inequality \eqref{eq:conj} holds for all $\sigma \in G$. Thus, if we write $G_i^{(1)},\ldots,G_i^{(d)}$ for the $d$ conjugates of $G_i$ in $K$, we then get that
$$
|\norm_{K/\Q}({\mathcal E}_i)|\le |\norm_{K/\Q}(E_i)|=|\norm_{K/\Q}(G_i)|\le \alpha^{11 d k {\sqrt{m}}},
$$
where the first inequality above follows because ${\mathcal E}_i$ divides $E_i$; hence $G_i$, and $E_i\ne 0$.
Multiplying the above inequalities for $i=1,\ldots,\ell$, we get that
\begin{eqnarray*}
E^{\ell} & = & |\norm_{K/\Q}(E)|=\left|\norm_{K/\Q}\left(E{\mathcal O}_{K}\right)\right|\le \left|\norm_{\K/\Q}\left(\prod_{i=1}^{\ell} {\mathcal E}_i\right)\right|\\
& \le & \left|\prod_{i=1}^{\ell} \norm _{K/\Q}({G}_i)\right|\le \alpha^{11 d\ell k {\sqrt{m}}},
\end{eqnarray*}
therefore
\begin{equation}
\label{eq:boundforE}
E\le \alpha^{11 k d {\sqrt{m}}}\le \alpha^{22 k \phi(s) {\sqrt{m}}}<\alpha^{22 k s {\sqrt{m}}}.
\end{equation}
In the above, we used that $d\le 2\phi(s)\le 2s$.

\medskip

We are now ready to estimate $A$.
We write
\begin{eqnarray*}
A_1 & := & \gcd(U_m, U_{n+k}^2-U_n^2);\\
A_2 & : = & \gcd(U_m, U_{n+k}^2+U_n^2);\\
A_3  & := & \gcd\left(U_m, \frac{U_{n+k}^6-U_n^6}{U_{n+k}^2-U_n^2}\right).
\end{eqnarray*}
Clearly, $A\le A_1 A_2 A_3$. We bound each of $A_1,~A_2,~A_3$. We first estimate $A_1$ and $A_2$ and deal with $A_3$ later. Write
\begin{eqnarray*}
U_n^2  &  = & \left(\frac{\alpha^n-\beta^n}{\alpha-\beta} \right)^2=\frac{\alpha^{2n}+2(-b)^n+\alpha^{-2n}}{(\alpha+b\alpha^{-1})^2};\\
U_{n+k}^2 &  = & \frac{\alpha^{2n+2k}+2(-b)^n (-b)^k+\alpha^{-2n-2k}}{(\alpha+b \alpha^{-1})^2}.
\end{eqnarray*}
Assume that $(-b)^k=1$. If $\zeta\in \{\pm i\}$, then $(\alpha^k-(-b)^k {\overline{\zeta}})/(\alpha^k-{\overline{\zeta}})=(\alpha^k+\zeta)/(\alpha^k-\zeta)$ is multiplicatively independent with $\alpha$ by Lemma \ref{lem:independence}.
The argument which lead to inequality~\eqref{eq:boundforE} shows that
\begin{equation}
\label{eq:estimateA2}
A_2\le \alpha^{11 k d_1 {\sqrt{m}}}\le \alpha^{44 k {\sqrt{m}}},
\end{equation}
where $d_1=4$ is the degree of the field ${\mathbb Q}(\alpha,i)$. To estimate $A_1$, we set $\gamma=-b\alpha^2$ and, using that $(-b)^k=1$, we find
\begin{eqnarray*}
U_{n+k}^2-U_n^2 & = & \frac{\alpha^{2n+2k}+\alpha^{-2n-2k}-\alpha^{2n}-\alpha^{-2n}}{(\alpha+b\alpha^{-1})^2}\\
& = & \alpha^{2-2n-k} \frac{(\gamma^{2n+k}-1)(\gamma^{k}-1)}{(\gamma-1)^2},\\
U_m & = & (-b \alpha)^{1-m} \left(\frac{\gamma^m-1}{\gamma-1}\right).
\end{eqnarray*}
In the ring of integers ${\mathcal O}={\mathcal O}_{K}$ of the quadratic field $K={\mathbb Q}(\alpha)$ consider the ideals
$$
\mathfrak{a}=\left(\frac{\gamma^m-1}{\gamma-1}, \frac{\gamma^{2n+k}-1}{\gamma-1}\right),\quad \mathfrak{b}=\left(\frac{\gamma^m-1}{\gamma-1},\frac{\gamma^k-1}{\gamma-1}\right).
$$
Clearly, $A_1\mid \mathfrak{a} \mathfrak{b}$, whence
$$
A_1^2= \norm_{{K/}{\mathbb Q}} (A_1)\le |\norm_{{K}/{\mathbb Q}}(\mathfrak{a})||\norm_{{K}/{\mathbb Q}}(\mathfrak{b})|.
$$
Clearly,
$$
|\norm_{{K}/{\mathbb Q}}(\mathfrak{b})|\le \left|\norm_{{K}/{\mathbb Q}}\left(\frac{(-b)^k\alpha^{2k}-1}{\alpha+b\alpha^{-1}}\right)\right|=|\norm_{{K}/{\mathbb Q}} (U_k)|<\alpha^{2k}.
$$
To estimate $|\norm_{{K}/{\mathbb Q}}(\mathfrak{a})|$, observe that $\mathfrak{a}=(\gamma^{d}-1)/(\gamma-1)$ by item 3 of Corollary \ref{cor:v_new}, where $d=\gcd(m,2n+k)$. Using the obvious inequality
$|\gamma^{-1}|\le 1/2$, we get that
$$
|\norm_{{\mathbb K}/{\mathbb Q}}(\mathfrak{a})|=\left|\frac{\gamma^{d}-1}{\gamma-1} \frac{\gamma^{-d}-1}{\gamma^{-1}-1}\right|\le 6 |\gamma|^{d}=6\alpha^{2d}<\alpha^{2d+4}.
$$
Hence,
${A_1\le \alpha^{d+k+2}}$.
It is important to note that $d\ne m$: otherwise we would have had  ${U_m\mid U_{n+k}^2-U_n^2}$, contradicting our hypothesis about the minimality of $s$.
Therefore $d$ is a proper divisor of $m$, showing that
\begin{equation}
\label{eq:estimateA1}
A_1\le \alpha^{m/2+k+2}.
\end{equation}
Thus, we have bounded $A_1$ and $A_2$ in the case $(-b)^k=1$.

The case ${(-b)^k=-1}$ can be treated analogously, but~$A_1$ and~$A_2$ swap roles. This time for $\zeta\in \{\pm 1\}$
the number $\frac{\alpha^k-(-b)^k {\overline{\zeta}}}{\alpha^k-\zeta} = \frac{\alpha^k+\zeta}{\alpha^k-\zeta}$ is multiplicatively independent of $\alpha$ by Lemma~\ref{lem:independence}, which implies the estimate
\begin{equation}
\label{eq:A1case2}
A_1\le \alpha^{22 k{\sqrt{m}}}.
\end{equation}
 Next, using that $(-b)^k=-1$, we find
 $$
 U_{n+k}^2+U_n^2=\alpha^{2-n-k} \frac{(\gamma^{2n+k}-1)(\gamma^k-1)}{(\gamma-1)^2},
 $$
 and arguing exactly as in the case $(-b)^k=1$, we get
 \begin{equation}
 \label{eq:A2case2}
 A_2\le \alpha^{m/2+k+2}.
 \end{equation}
 Hence, we get that both in case $(-b)^k=1$ and in case $(-b)^k=-1$, we have
 \begin{equation}
 \label{eq:A1A2}
 A_1A_2\le \alpha^{m/2+k+2+44k{\sqrt m}}.
 \end{equation}
Finally, for $A_3$, we note that by Lemma~\ref{lem:independence}, unless $\alpha=2+{\sqrt{3}}$, we have that $\frac{\alpha^k-(-b)^k {\overline{\zeta}}}{\alpha^k-\zeta}$ is multiplicatively independent of $\alpha$ for $\zeta\in \{\pm \omega, \pm \omega^2\}$.
Thus, writing
$$
A_{3,\pm }=\gcd(A_3, U_{n+k}^2\pm U_{n+k} U_n+U_n^2),
$$
we get, by arguing in the field ${\mathbb Q}(\alpha,e^{2\pi i/3})$ of degree $4$ as we did in order to prove inequality \eqref{eq:boundforE}, that
\begin{equation}
\label{eq:A3pm}
A_{3,\pm }\le \alpha^{44 k {\sqrt{m}}},
\end{equation}
which leads to
\begin{equation}
\label{eq:A3}
A_3\le A_{3,+} A_{3,-}\le  \alpha^{88 k {\sqrt{m}}}.
\end{equation}
So, let us assume that $(a,b,k)=(4,1,1)$, so $\alpha=2+{\sqrt{3}}$. Note that since $U_t\equiv t\pmod 2$, it follows that $U_{n+k}^s-U_n^s=U_{n+1}^s-U_n^s$ is odd and a multiple of $U_m$, therefore $m$ is odd.  For $\zeta\in \{\omega, \omega^2\}$, we have that $\frac{\alpha^k-(-b)^k {\overline{\zeta}}}{\alpha^k-\zeta}=\frac{\alpha-{\overline{\zeta}}}{\alpha-\zeta}$ are multiplicatively independent of $\alpha$, which leads, by the previous argument, to
\begin{equation}
\label{eq:A3p}
A_{3,+}\le \alpha^{44 k {\sqrt{m}}}.
\end{equation}
As for $A_{3,-}$, since
$$
U_{n+1}^2-U_{n+1} U_n+U_n^2=V_{2n+1}/4,
$$
we have that
$$
A_{3,-}\mid \gcd(U_m, V_{2n+1})=1,
$$
where the last equality follows easily from the fact that~$m$ is and ${2n+1}$ are both odd (see $(iii)$ of the Main Theorem in~\cite{McDaniel}). Together with~\eqref{eq:A3p}, we infer that inequality~\eqref{eq:A3} holds in this last case as well. Together with~\eqref{eq:A1A2}, we get that the inequality
\begin{equation}
\label{eq:boundforA}
A\le A_1A_2A_3\le \alpha^{m/2+k+2+132 k {\sqrt{m}}}
\end{equation}
holds in all instances.

\medskip

Inequality~\eqref{eq:lowup} together with estimates~\eqref{eq:boundforD},~\eqref{eq:boundforA} and~\eqref{eq:boundforE}, give
$$
\alpha^{m-2}\le U_m=DAE\le \alpha^{6s+3+\log m/\log \alpha+m/2+k+2+(132k+22ks) {\sqrt{m}}}.
$$
Since $s\ge 3$, we have $132+22s\le 66s$. Since also $1/\log \alpha<3$, we get
$$
m/2\le (6s+7+3\log m+k)+66s k{\sqrt{m}}.
$$
Since $m\ge 10000$, one checks that $6s+7+3\log m+k<ks{\sqrt{m}}$. Hence,
\begin{equation}
\label{eq:last}
m\le 134 ks {\sqrt{m}},
\end{equation}
which leads to the desired inequality \eqref{eq:mainthm_new}.

\section{Comment}

One may wonder if one can strengthen our main result Theorem \ref{thm:main} in such a way as to include also the instances $s\in \{1,2,4\}$ maybe at the cost of eliminating finitely many exceptions in the pairs $(a,k)$. The fact that this is not so
follows from the formulae:
\begin{itemize}
\item[$(i)$] $U_{n+k}-U_n=U_{n+k/2} V_{k/2}$ for all $n\ge 0$ when $b=1$ and $2\| k$;
\item[$(ii)$] $U_{n+k}+U_n=U_{n+k/2} V_{k/2}$  for all $n\ge 0$ when $b=1$ and $4\mid k$ or when $b=-1$ and $k$ is even;
\item[$(iii)$] $U_{n+k}^2+U_n^2=U_{2n+k} U_k$  for all $n\ge 0$ when $b=1$ and $k$ is odd,
\end{itemize}
which can be easily proved using the Binet formulas~\eqref{eq:Binet}. Thus, taking $m=n+k/2$ (for $k$ even) and $m=2n+k$ for $k$ odd and $b=1$, we get that divisibility \eqref{eq:Untos} always holds with some $s\in \{1,2,4\}$. We also note the ``near-miss'' $U_{4n+2}\mid 4(U_{n+1}^6-U_n^6)$ for all $n\ge 0$ if $(a,b,k)=(4,-1,1)$.

\section*{Acknowledgements}

This work was done during a visit of T.~K., A.~P. and P.~S. to the School of Mathematics of the Wits University in September 2015. They thank the School for hospitality and support and Kruger National Park for excellent working conditions. F.~L. thanks Professor Igor Shparlinski for some useful suggestions. Yu.~B. was partially supported by  Max-Planck-Institut f\"ur Mathematik at Bonn. T.~K. was partially supported by the Hubei provincial Expert Program in China. A.~P. was partly financed by Proyect DIUV-REG N$^{\circ}$~25-2013.


\begin{thebibliography}{999}

\bibitem{BHV} Yu. Bilu, G. Hanrot, P. M. Voutier, ``Existence of primitive divisors of Lucas and Lehmer numbers, with an appendix by M. Mignotte", {\it J. Reine Angew. Math.\/} {\bf 539} (2001), 75--122.

\bibitem{KLT}  T. Komatsu, F. Luca, Y. Tachiya, ``On the multiplicative order of $F_{n+1}/F_n$ modulo $F_m$", {\em Integers} {\bf 12B} (2012/13), Integers Conference 2011 Proc., \#A8, 13pp.

\bibitem{McDaniel} W. L. McDaniel, ``The G.C.D. in Lucas sequences and Lehmer number sequences", {\it Fibonacci Quart.\/}  {\bf 29} (1991), 24--29.


\end{thebibliography}
\end{document}